\documentclass[11pt, a4paper]{amsart}

\usepackage[T1]{fontenc}
\usepackage[linktocpage]{hyperref} % For hyperlinks in the PDF [linktocpage]
	\hypersetup{colorlinks=true,linkcolor=blue,citecolor=blue,filecolor=magenta,urlcolor=blue}      
\usepackage{geometry}       
\usepackage[english]{babel}  
\usepackage{amsmath,amsfonts,amssymb,amsthm}
\usepackage{dsfont}
\usepackage{tikz}%,tkz-graph,tikz-cd}
	\usetikzlibrary{calc}
	\usetikzlibrary{positioning}

\newtheorem{theorem}{Theorem}[section]
\newtheorem{proposition}[theorem]{Proposition}

\newtheorem{corollary}[theorem]{Corollary}

\theoremstyle{definition}
\newtheorem{definition}[theorem]{Definition}                  
\newtheorem{example}[theorem]{Example}

\theoremstyle{remark}
\newtheorem{rmk}[theorem]{Remark}

\newcommand{\CC}{\mathds{C}}

\newcommand{\QQ}{\mathds{Q}}
\newcommand{\FF}{\mathds{F}}
\newcommand{\PP}{\mathds{P}}
\newcommand{\A}{\mathcal{A}}
\newcommand{\B}{\mathcal{B}}
\newcommand{\R}{\mathcal{R}}
\newcommand{\C}{\mathcal{C}}
\newcommand{\M}{\mathcal{M}}
\newcommand{\ML}{\mathcal{ML}}
\renewcommand{\P}{\mathcal{P}}

\renewcommand{\L}{\mathcal{L}}

\renewcommand{\epsilon}{\varepsilon}
\newcommand{\osim}{\stackrel{\ord}{\sim}}

\newcommand{\ord}{\text{ord}}

\DeclareMathOperator{\sing}{Sing}

\DeclareMathOperator{\PGL}{PGL}

\renewcommand{\line}[5]{ 
	\draw ($ #1 - 1/100*#3*#2 + 1/100*#3*#1 $)  -- ($ #2 + 1/100*#4*#2 - 1/100*#4*#1 $);
	\node at ($ #1 - 1.3/100*#3*#2 + 1.3/100*#3*#1 $)  {#5};
}

\begin{document}

\title[On the non-connectivity of moduli spaces of arrangements]{On the non-connectivity of moduli spaces of arrangements: the splitting-polygon structure}

\author[B. Guerville-Ball\'e]{Beno\^it Guerville-Ball\'e}
\address{
	Institute of Mathematics for Industry, Kyushu University \\
	Motooka 744, Nishi-ku, Fukuoka 819-0395, Japan}
\email{benoit.guerville-balle@math.cnrs.fr }

\thanks{}				% Pour les financements

		% Code AMS

\begin{abstract}
	Questions that seek to determine whether a hyperplane arrangement property, be it geometric, arithmetic or topological, is of a combinatorial nature (that is determined by the intersection lattice) are abundant in the literature. To tackle such questions and provide a negative answer, one of the most effective methods is to produce a counterexample. To this end, it is essential to know how to construct arrangements that are lattice-equivalent. The more different they are, the more efficient it will be.
		
	In this paper, we present a method to construct arrangements of complex projective lines that are lattice-equivalent but lie in distinct connected components of their moduli space. To illustrate the efficiency of the method, we apply it to reconstruct all the classical examples of arrangements with disconnected moduli spaces: MacLane, Falk-Sturmfels, Nasir-Yoshinaga and Rybnikov. Moreover, we employ this method to produce novel examples of arrangements of eleven lines whose moduli spaces are formed by four connected components.
\end{abstract}

\maketitle

%%%%%%%%%%%%%%%%%%%%%%%%%%%%%%%%%
%%%%%%%%%%%%%%%%%%%%%%%%%%%%%%%%%
\section*{Introduction}
%%%%%%%%%%%%%%%%%%%%%%%%%%%%%%%%%
%%%%%%%%%%%%%%%%%%%%%%%%%%%%%%%%%

As mentioned in the abstract, the questions related to the combinatorial nature of some properties of a hyperplane arrangement are numerous in the literature. If some of them have been solved by the affirmative, as for the number of chambers of a real arrangement~\cite{Zas}, the cohomology ring of the complement~\cite{OrlSol}, the rank of the lower central series quotients of the fundamental group of its complement~\cite{Falk:minimal} or the deletion and addition-deletion theorems of free arrangements~\cite{Abe:deletion,Abe:addition-deletion}; some others obtained a negative answer, as for the embedded topology of a complex arrangements or the fundamental group of its complements, see~\cite{Ryb, Gue:ZP, ACGM:arithmetic, Gue:LLN}, (also negative for the smaller class of real complexified arrangements~\cite{ACCM:topology, GueViu:configurations}), the torsion of the lower central series quotients~\cite{ AGV:torsion} or the existence of unexpected curves~\cite{GueViu:configurations}. Naturally, the number of problems which are still open (or conjectural) is larger; like the famous Terao's conjecture~\cite{Ter, OrlTer}, the combinatorial nature of the characteristic varieties~\cite{Lib} or of the homology of the Milnor fiber~\cite[Problem 4.5]{FalkRan:homotopy}, to name some but a few.
	
The aim of this paper is to provide a method to construct line arrangements with non-connected moduli spaces, and thus of lattice-equivalent arrangements which cannot be deformed one into the other continuously and equisingularly. The method starts with a line arrangement $\A$ in which we pick $r$ lines (called the \emph{support}) and $r$ singular points (called the \emph{pivot-points}) that together form a \emph{plinth} of $\A$. On this plinth, we will add $r$ lines, each one passing through a single pivot-point. These lines form a \emph{splitting-polygon} if the corners of the polygon lie on the lines of the support. In Theorem~\ref{thm:SP}, we prove that the previous construction produces two arrangements $\A^{\lambda_1}$ and  $\A^{\lambda_2}$ which lie in different connected components of their moduli space. In  Theorem~\ref{thm:irreducible}, we give a sufficient algebraic condition on the plinth to the existence of splitting-polygons. This method also provides a combinatorial pattern, called the \emph{splitting-polygon pattern}, which is a strong indicator of a potential disconnected moduli space of a line combinatorics.

We illustrate the relevance of this method by reconstructing of the all the classical arrangements which have a non-connected moduli space: the MacLane~\cite{Mac}, the Nazir-Yoshinaga~\cite{NazYos}, the Falk-Sturmfels~\cite{CohSuc:braid} and the Rybnikov~\cite{Ryb} arrangements. As a final illustration of this method, we construct several arrangements of eleven lines which have moduli spaces form by four connected components. To construct these arrangements, we add a splitting-triangle on different plinths of the MacLane arrangements. Some of these examples have the additional topological property to have non-isomorphic fundamental groups of their complements (see Theorem~\ref{thm:pi_1}).
	
The paper is organized as follows: in Section~\ref{sec:moduli}, we recall some classical definitions related to line arrangements such as: the line combinatorics, the realization space and the moduli space. Section~\ref{sec:SPS} is devoted to the main results with the construction and proof of the method. The two last sections --Section~\ref{sec:examples} and~\ref{sec:successive}-- present applications of the method with the classical examples and the construction of new arrangements of eleven lines with disconnected moduli spaces.

%%%%%%%%%%%%%%%%%%%%%%%%%%%%%%%%%
%%%%%%%%%%%%%%%%%%%%%%%%%%%%%%%%%
\section{Moduli space}\label{sec:moduli}
%%%%%%%%%%%%%%%%%%%%%%%%%%%%%%%%%
%%%%%%%%%%%%%%%%%%%%%%%%%%%%%%%%%

The purpose of this section is to recall some classical definitions associated to the combinatorics of line arrangements and their realization space. 

\begin{definition}	\label{def:combinatorics}
	A \emph{line combinatorics} $(\L,\P)$ is the data of a finite set $\L$ and a subset $\P$ of the power set of $\L$ which verify:
	\begin{itemize}
		\item for all $P\in\P$, $\# P \geq 2$,
		\item for all $L_1, L_2 \in \L$, it exists a unique $P\in\P$ such that $L_1\in P$ and $L_2\in P$.
	\end{itemize}
	An \emph{ordered line combinatorics} is a line combinatorics with a total order on $\L$.
\end{definition}

Two line combinatorics $\C_1=(\L_1,\P_1)$ and $\C_2=(\L_2,\P_2)$ are equivalent if it exists a one-to-one correspondence $\phi$ from $\L_1$ to $\L_2$ such that for all $P\in\P_1$, we have $\phi(P)=\{ \phi(L) \mid L\in P \} \in \P_2$ (or equivalently $\phi(\P_1)=\P_2$). If $\C_1$ and $\C_2$ are ordered line combinatorics and if $\phi$ respects this order, then $\C_1$ and $\C_2$ are equivalent ordered line combinatorics. 

Let $\A$ be a line arrangement and let $\sing(\A)$ be its set of singular points\footnote{In all this paper, the singular points of a line arrangements are given as the set of lines of $\A$ which pass through the singular points. This allows to fit with Definition~\ref{def:combinatorics}.}, the couple $(\A,\sing(\A))$ is a line combinatorics, called the \emph{combinatorics of $\A$}. Analogously, we can define the \emph{ordered combinatorics} of an ordered line arrangement. They will be denoted $\C(\A)$ and $\C^{\ord}(\A)$ respectively\footnote{If there is no ambiguity about the arrangement $\A$ considered, we will sometime omit the $\A$ and denote the combinatorics by $\C$ and $\C^{\ord}$.}. Two arrangements (resp. ordered arrangements) $\A_1$ and $\A_2$ are $\C$-equivalent, also called lattice-equivalent, (resp. $\C^{\ord}$-equivalent) if their combinatorics (resp. ordered combinatorics) are equivalent. These equivalences are denoted $\sim$ and $\osim$ respectively. Conversely, a \emph{realization} of a combinatorics $\C$ is a line arrangement $\A$ such that~$\C(\A)\sim \C$.

\begin{definition}
	The \emph{realization space} $\R(\C)$ of a line combinatorics $\C=(\L,\P)$ is the set of $\C$-equivalent line arrangements, that is to say
	\begin{equation*}
		\R(\C) = \{ \A \mid \C(\A)\sim \C \}.
	\end{equation*}
	We define accordingly the \emph{ordered realization space} of an ordered line combinatorics, and we denote it by~$\R^{\ord}(\C)$.
\end{definition}

The group $\PGL_3(\CC)$ naturally acts on $\R(\C)$ and $\R^{\ord}(\C)$. So, the following definition is natural.

\begin{definition}
	The \emph{moduli space} $\M(\C)$ of a line combinatorics $\C$ is the quotient of the realization space of $\C$ by the action of $\PGL_3(\CC)$. The notion of \emph{ordered moduli space} is defined accordingly.
	
	If $\B$ is an element of $\R(\C)$, then the connected component of $\M(\C)$ which contains the class of $\B$ is denoted by $\M(\C)^\B$.
\end{definition}

By an abuse of notation and denomination, we call \emph{moduli space of a line arrangement $\A$}, and denote it by $\M(\A)$, the moduli space $\M(\C(\A))$ of its combinatorics.

\begin{definition}
	Let $\A$ be a line arrangement and $\B$ a sub-arrangement of $\A$. The \emph{realization space of $\C(\A)$ relative to $\B$} is the set of $\C(\A)$-equivalent line arrangements which contains $\B$. It is denoted $\R(\C(\A);\B)$ (or $\R(\A;\B)$). We define accordingly the \emph{moduli space of  $\A$ relative to $\B$}, denoted by~$\M(\A;\B)$, and their ordered equivalents denoted with a superscript $\ord$.
	
	Let $B$ be a subset of $\M(\B)$ (usually $\{\B\}$, $\M(\B)^\B$ and $\M(\B)$). The \emph{moduli space $\M(\A)$ splits over $B$} if $ 2 \leq \# \M(\A;\B')$ for all~$\B'\in B$.
\end{definition}

%%%%%%%%%%%%%%%%%%%%%%%%%%%%%%%%%
%%%%%%%%%%%%%%%%%%%%%%%%%%%%%%%%%
\section{The splitting-polygon structure}\label{sec:SPS}\mbox{}
%%%%%%%%%%%%%%%%%%%%%%%%%%%%%%%%%
%%%%%%%%%%%%%%%%%%%%%%%%%%%%%%%%%

In order to define the \emph{splitting-polygon} structure, we need to introduce the notion of \emph{plinth} of an arrangement $\A$. It will be describe the position of the "anchor points" on $\A$ of the splitting-polygon.

%%%%%%%%%%%%%%%%%%%%%%%%%%%%%%%%%
\subsection{Plinth of an arrangement}\mbox{}
%%%%%%%%%%%%%%%%%%%%%%%%%%%%%%%%%

Let $\C=(\L,\P)$ be a line combinatorics, we denote by $\P_{\geq k}$ the subset $\{P\in\P \mid \# P \geq k \}$ of $\P$, and by $\P_k$ the subset $\P_{\geq k}\setminus \P_{\geq (k-1)}$. 

\begin{definition}
	Let $\C=(\L,\P)$ be a line combinatorics and let $3 \leq r \leq \# \A$. A \emph{plinth} $\Psi$ in $\C$ is form by two tuples\footnote{In this paper, tuples are sets with total orders, and they are denoted using parentheses. We always consider the order given by writing order of the set (i.e. in the tuple $(a,b,c,d)$, we have $a<b<c<d$).}: the \emph{support} $S=(S_1,\dots,S_r) \subset \L$ and the \emph{pivot-points} $(P_1,\dots,P_r)\subset\P$ such that, for each pivot-point $P_i$, we have $S_i\notin P_i$ and $S_{i+1}\notin P_i$.
	
A line arrangement $\A$ is said to have a plinth if its combinatorics does. 
\end{definition}

In the next section, we will need to tighten up the plinth of an arrangement. So, let us introduce the notion of rigid projective system.

\begin{definition}\label{def:RPS}
	Let $\A=\{L_1,\dots,L_n\}$  be a line arrangement and let $3 \leq r \leq n$. Two subsets: $\{L_{i_1},\dots,L_{i_r}\}$ of $\A$ and $\{P_{j_1},\dots,P_{j_r}\}$ of $\sing(\A)$, form a \emph{rigid projective system of $\M^\ord(\A)^\A$} (resp. \emph{of $\M^\ord(\A)$}) if for all arrangement $\A'=\{L_1',\dots,L_n'\}$ in $\M^\ord(\A)^\A$ (resp. in $\M^\ord(\A)$) there exists a projective transformation $\tau\in\PGL_3(\CC)$ such that $\tau(L_{i_j})=L'_{i_j}$ and $\tau(P_{i_j})=P'_{i_j}$, for $j\in\{1,\dots,r\}$. 
\end{definition}

It's obvious that a rigid projective system of $\M^\ord(\A)$ is also a rigid projective system of $\M^\ord(\A)^\A$. Nevertheless, the converse does not seem clear to us.

\begin{example}\label{ex:plinth}
	We consider the arrangement $\A=\{L_1,\dots,L_6\}$ form by 6 lines with one quadruple points $\{L_1,L_2,L_3,L_4\}$, one triple points $\{L_1,L_5,L_6\}$, and all the other points are double points (see Figure~\ref{fig:arr_1}). The tuples $S=(L_1,L_2,L_5)$ and $P=(\{L_2,L_6\},\{L_3,L_6\},\{L_3,L_5\})$ form a rigid projective system; while the tuples $S$ and $P'=(\{L_4,L_6\},\{L_3,L_6\},\{L_3,L_5\})$ do not. Indeed, if a projective transformation fixes $S$ and $\{L_3,L_6\}$, then it also fixes $\{L_2,L_6\}$ and $\{L_3,L_5\}$, but it does not fix $\{L_4,L_6\}$ since the line $L_4$ can be any line of the line pencil define by the quadruple point.
\end{example}

\begin{figure}[h!]
	\begin{tikzpicture}
		\begin{scope}[yscale=0.66]
			\draw (0,5) -- (0,-1);
			\draw (-1/6,5) -- (1,-1);
			\draw (-2/6,5) -- (2,-1);
			\draw (-3/6,5) -- (3,-1);
			\draw (-1,0) -- (4, 0);
			\draw (-1,-1/2) -- (4, 2); 
			\node[below] at (0,-1) {$L_1$};
			\node[below] at (1,-1) {$L_2$};
			\node[below] at (2,-1) {$L_3$};
			\node[below] at (3,-1) {$L_4$};
			\node[left] at (-1,0) {$L_5$};
			\node[below] at (-1,-1/2) {$L_6$};
		\end{scope}
	\end{tikzpicture}
	\caption{Arrangement with one quadruple and one tripe point. \label{fig:arr_1}}
\end{figure}
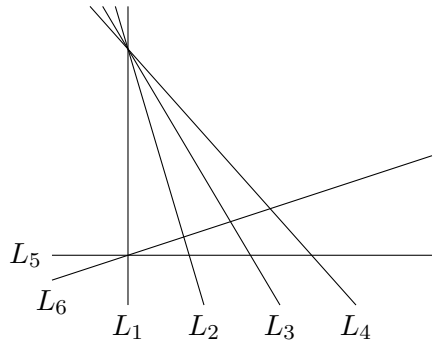

The following proposition follows directly from the definitions of the moduli space and of a rigid projective system.

\begin{proposition}\label{propo:rigidity}
	If $\A$ is a line arrangement such that $\M^\ord(\A)^\A$ has dimension $0$, then any subsets $\{L_{i_1},\dots,L_{i_r}\}$ of $\A$, and $\{P_{j_1},\dots,P_{j_r}\}$ of $\sing(\A)$ form a rigid projective system of $\M^\ord(\A)^\A$.
\end{proposition}

%%%%%%%%%%%%%%%%%%%%%%%%%%%%%%%%%
\subsection{The splitting-polygon structure}\label{sec:construction}\mbox{}
%%%%%%%%%%%%%%%%%%%%%%%%%%%%%%%%%

Let $\A$ be a line arrangement such that, for a fixed $3 \leq r\leq \# \A$, the lines $(S_1,\dots,S_r)\in\A$ and the singular points $(P_1,\dots,P_r)\in\sing(\A)$ form a plinth $\Psi$.  

Let $Q^\lambda_1$ be a generic point of $S_1$ which is determined by a parameter $\lambda\in\CC$. We define $E^\lambda_1$ as the line which passes through $Q^\lambda_1$ and $P_1$. This line intersects $S_2$ in a point denoted $Q^\lambda_2$. We define recursively the lines $E^\lambda_i$ as the lines which pass through $Q^\lambda_i$ and $P_i$, and the points $Q^\lambda_{i+1}$ as the intersection points of $E^\lambda_i$ and $S_{i+1}$. At the end, we define $R^\lambda_1$ as the intersection point of $E^\lambda_r$ and~$S_1$ (see Figure~\ref{fig:SP} for an illustration of the construction when $r=3$). We denote by $A^\lambda$ the arrangement $\A \cup \{E^\lambda_1,\dots,E^\lambda_r\}$.

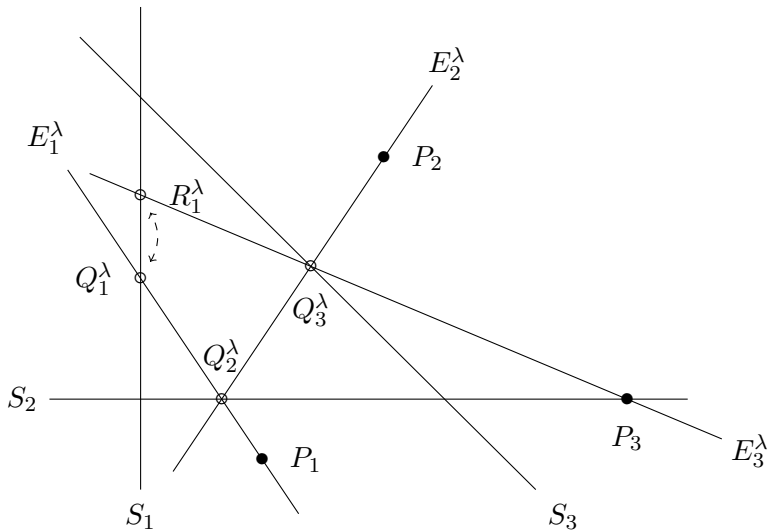
\begin{figure}[h!]
	\begin{tikzpicture}
		\begin{scope}[scale=0.8]
			\node (C1) at (0,0) {};
			\node (C2) at (5,0) {};
			\node (C3) at (0,5) {};
			\line{(C1)}{(C3)}{30}{30}{$S_1$};
			\line{(C1)}{(C2)}{30}{80} {$S_2$};
			\line{(C2)}{(C3)}{30}{20}{$S_3$};
			\node[label=left:$Q_1^\lambda$] (Q1) at (0,2) {$\circ$};
			\node[label=right:$P_1$] (P1) at (2,-1) {$\bullet$};
			\line{(Q1)}{(P1)}{60}{30}{$E_1^\lambda$};
			\node[label=above:$Q_2^\lambda$] (Q2) at (4/3,0) {$\circ$};
			\node[label=right:$P_2$] (P2) at (4,4) {$\bullet$};
			\line{(P2)}{(Q2)}{30}{30}{$E_2^\lambda$};
			\node[label=below:$Q_3^\lambda$] (Q3) at (2.8,5-2.8) {$\circ$};
			\node[label=below:$P_3$] (P3) at (8,0) {$\bullet$};
			\line{(P3)}{(Q3)}{30}{70}{$E_3^\lambda$};
			\node[label=right:$R_1^\lambda$] (R1) at (0,3.3666666) {$\circ$};
			\draw[<->, dashed] (Q1) to[out=60, in=300] (R1);
		\end{scope}
	\end{tikzpicture}
	\caption{Construction of a splitting-polygon for $r=3$. The pivot-points are noted with a $\bullet$, while the corners of the polygon are represented by a $\circ$.  \label{fig:SP}}
\end{figure}

\begin{definition}\label{def:SP}
	The tuple $E^\lambda=(E^\lambda_1,\dots,E^\lambda_r)$ forms a \emph{splitting-polygon} on the plinth $\Psi$ if:
	\begin{enumerate}
		\item $Q^\lambda_1 = R^\lambda_1$,
		\item for all $i\in\{1,\dots,n\}$, we have $E^\lambda_i \notin \A$,
		\item each line $E^\lambda_i$ contains $ \# \A + r - \# P_i - 2 $ singular points in $\A^\lambda$.
	\end{enumerate}
\end{definition}

The first condition is that $(E^\lambda_1\dots,E^\lambda_r)$ form a polygon whose $i$th corners (i.e. the intersection point of the two successive edges $E_i^\lambda$ and $E_{i+1}^\lambda$) is on the $i$th line of support of the plinth~$\Psi$; while second and third conditions are equivalent to the fact that, for all $i\in\{1,\dots,r\}$, the point $Q_i^\lambda$ is in $\sing(\A)_{3}$, and the line $E^\lambda_i$ intersects $\A^\lambda$ generically excepts in $Q_{i-1}$, $Q_i$ and $P_i$. If $E^\lambda$ is a splitting-polygon on the plinth $\Psi$ of $\A$, then the combinatorics of $\A\cup E^\lambda$ is denoted by $\C_\Psi$. In Figure~\ref{fig:SP}, the lines of $E^\lambda$ form a splitting-triangle when $Q_1^\lambda$ and $R_1^\lambda$ coincide (it is represented by the dashed arrow).

\begin{theorem}\label{thm:SP}
	Let $(S_1,\dots,S_r)$ and $(P_1,\dots,P_r)$ be a plinth $\Psi$ of an arrangement $\A$ which form a rigid projective system of  $\M^\ord(\A)^\A$. If $E^{\lambda_1}$ and $E^{\lambda_2}$ are two distinct splitting-polygons on $\Psi$, then $\M^\ord(\C_\Psi)$ splits over $\M^\ord(\A)^\A$. More precisely, $\# \M^\ord(\C_\Psi;\M^\ord(\A)^\A) = 2$, and $\A^{\lambda_1}$ and $\A^{\lambda_2}$ are in different connected components of $\M^\ord(\C_\Psi)$.
\end{theorem}

\begin{proof}
	According to Definition~\ref{def:SP}, the arrangement $\A^\lambda$ admits a splitting-polygon on the plinth $\Psi$ only if the points $Q^\lambda_1$ and $R^\lambda_1$ are equal, or equivalently if the lines $S_1$, $E^\lambda_1$ and $E^\lambda_r$ are concurrent. Let $\Delta_{\Psi}$ be the determinant of the matrix formed by the coefficients of these three lines equations. Since the equation of $S_1$ is independent on $\lambda$, and since the coefficients of $E^\lambda_1$ and $E^\lambda_r$ are linear in $\lambda$, then $\Delta_{\Psi}$ is a polynomial of degree at most 2 of $\CC[\lambda]$. It follows that it exists at most 2 splitting-polygons on $\Psi$. This upper-bound is an equality since $E^{\lambda_1}$ and $E^{\lambda_2}$ are two distinct splitting-polygons on the plinth $\Psi$ (that is $\lambda_1$ and $\lambda_2$ are two distinct roots of $\Delta_{\Psi}$ in $\CC$).

	Since $E^{\lambda_1}$ and $E^{\lambda_2}$ are two splitting-polygons on the plinth $\Psi$, then $\A^{\lambda_1}$ and $\A^{\lambda_2}$ are $\C_\Psi^\ord$-equivalent, and so are representatives of two classes in $\M^\ord(\C_\Psi)$ (a priori, not necessarily distinct nor in different connected components). So, let assume that $\A^{\lambda_1}$ and $\A^{\lambda_2}$ are in the same connected component of $\M^\ord(\C_\Psi;\M^\ord(\A)^\A)$. That is, it exists a continuous path of $\C_\Psi^\ord$-equivalent arrangements $\A_t$, with $t \in [0,1]$, such that $\A_0=\A^{\lambda_1}$ and  $\A_1=\A^{\lambda_2}$. Since the elements of $\M^\ord(\C_\Psi)$ are considered up to projective transformation, and since the plinth $\Psi$ forms a rigid projective system then we can consider that both, the support and the pivot-points of $\Psi$, are constant along $\A_t$. By the first paragraph of this proof, we know that there are a finite number of possible splitting-polygons on $\Psi$. Thus, since $\Psi$ is constant along $\A_t$, then the splitting-polygon cannot vary along the path $\A_t$. This induces an incompatibility with the assumption that $\A^{\lambda_1}$ and $\A^{\lambda_2}$ are in the same connected components of $\M^\ord(\C_\Psi;\M^\ord(\A)^\A)$. Due to the fact that it exists exactly two splitting-polygons on $\Psi$, we deduce that $\# \M^\ord(\C_\Psi;\M^\ord(\A)^\A) = 2$.
	
	If there is a continuous path between $\A^{\lambda_1}$ and $\A^{\lambda_2}$ in $\M^\ord(\C_\Psi)$ then it has to stay inside the relative moduli space $\M^\ord(\C_\Psi;\M^\ord(\A)^\A)$ since $\A^{\lambda_1}$ contains $\A$ (by construction). Thus, the previous paragraph produces the obstruction.
\end{proof}

As a direct consequence of the proof of Theorem~\ref{thm:SP}, we have the following proposition.

\begin{proposition}\label{prop:dimension}
	Let $(S_1,\dots,S_r)$ and $(P_1,\dots,P_r)$ be a plinth $\Psi$ of an arrangement $\A$ which form a rigid projective system of  $\M^\ord(\A)$. If $\M^\ord(\C_\Psi)$ is not empty then the moduli spaces $\M^\ord(\A)$ and $\M^\ord(\C_\Psi)$ have the same dimension.
\end{proposition}

A way to predict the existence of two distinct splitting-polygons is to study the polynomial $\Delta_{\Psi}$. The lines equations of $\A$ have their coefficients in $\CC$. Nevertheless, it is possible to consider them in a field extension of $\QQ$ (since they are solutions of polynomial equations with integral coefficients). We denote by $\FF$ such a definition field of $\A$. By definition field of $\A$, we mean a field which contains all the lines coefficients of $\A$ (it is worth to notice that it can differ from a field which contains the coefficients of an equation of $\A$\footnote{There exist arrangements with rational equation but whose lines have non-real equations.}). Of course, if $\FF'$ is an extension of $\FF$ then it is also a definition field of $\A$. 

\begin{theorem}\label{thm:irreducible}
	Let $\FF$ be a definition field of $\A$, and let $(S_1,\dots,S_r)$ and $(P_1,\dots,P_r)$ be a plinth $\Psi$ of $\A$ which form a rigid projective system of  $\M^\ord(\A)^\A$.  If $\Delta_{\Psi}$ is non-trivial and is irreducible in $\FF[\lambda]$ then $\A^{\lambda_1}$ and $\A^{\lambda_2}$ are $\C^\ord_\Psi$-equivalent and the moduli space $\M^\ord(\C_\Psi)$ splits over $\M^\ord(\A)^\A$. More precisely, $\# \M^\ord(\C_\Psi;\M^\ord(\A)^\A) =2$, and $\A^{\lambda_1}$ and $\A^{\lambda_2}$ are in distinct connected-components of $\M^\ord(\C_\Psi)$.
\end{theorem}

\begin{proof}
	Let $\overline{\FF}$ be the decomposition-field of $\Delta_{\Psi}$. Since $\Delta_{\Psi}$ is irreducible in $\FF$ and not zero then $\lambda_1$ and $\lambda_2$ (the two roots of $\Delta_{\Psi}$) are in $\overline{\FF}\setminus\FF$. It follows that the definition field of the lines $E^{\lambda_i}_1,\dots,E^{\lambda_i}_r$, for $i\in\{1,2\}$, is $\overline{\FF}$ (and cannot be $\FF$). This and the construction imply that the coordinates of the intersection points of the $E^{\lambda_i}_j$ with the lines of $\A$ are necessarily in $\overline{\FF}$ (and cannot be in $\FF$), except for the intersection points with the lines passing through $P_i$. Thus, except the $P_i$'s and the $Q_i$'s, only double points will be created when we add the $E^{\lambda_i}_j$ to the arrangement $\A$. In other words, $E^{\lambda_1}$ and $E^{\lambda_2}$ are two distinct splitting-polygons on $\Psi$. The end of the theorem follows from Theorem~\ref{thm:SP}.
\end{proof}

As a straightforward consequence of the previous proof, we have the following corollary.

\begin{corollary}
	Let $(S_1,\dots,S_r)$ and $(P_1,\dots,P_r)$ be a plinth $\Psi$ of an arrangement $\A$ which form a rigid projective system of  $\M^\ord(\A)$. We assume that $\M^\ord(\A)$ admits $d$ connected component, and we consider $\A_1,\dots,\A_d$ representatives of each component, and let $\FF$ be a definition field of the $\A_i$'s. We denote by $\Delta_{\Psi}^1,\dots,\Delta_{\Psi}^d$ the polynomials constructed as before by considering the arrangements $\A_1,\dots,\A_d$. 
	
	If one of the polynomial $\Delta_{\Psi}^i$ is non-trivial and irreducible in $\FF[\lambda]$ then all the $\Delta_{\Psi}^i$'s are also non-trivial and irreducible in $\FF[\lambda]$. Thus, the arrangements $\A_i^{\lambda_j^i}$ are $\C^\ord_\Psi$-equivalent and  $\M^\ord(\C_\Psi)$ splits over $\M^\ord(\A)$, where $\lambda_1^i$ and $\lambda_2^i$ are the two complex roots of $\Delta_\Psi^i$. More precisely, we have $\# \M^\ord(\C_\Psi) = 2 d$ and the arrangements $\A_i^{\lambda_j^i}$, for $i\in\{1,\dots,d\}$ and $j\in\{1,2\}$, are representatives of each connected components.
\end{corollary}

\begin{proof}
	Since the plinth $\Psi$ form a rigid projective system of $\M^\ord(\A)$ then the equations of the $E_k^{i,\lambda}$ and $E_k^{j,\lambda}$  are equal for all $i,j\in\{1,\dots,d\}$. In particular, this implies that $\Delta_{\Psi}^1 = \dots = \Delta_{\Psi}^d$. The end follows from Theorem~\ref{thm:irreducible}.
\end{proof}

Unfortunately, the implication obtained in Theorem~\ref{thm:irreducible} is not an equivalence. Indeed, in Section~\ref{sec:successive}, we prove that the Rybnikov arrangements can be constructed using the splitting-polygon method twice. It appears that the polynomial $\Delta_{\Psi}$ of the second splitting-polygon is reducible in $\FF$ and that the moduli space splits. In fact the hypothesis: $\Delta_{\Psi}$ is irreducible, is only used to prove that $\A^{\lambda_1}$ and $\A^{\lambda_2}$ are $\C_\Psi^\ord$-equivalent. Nevertheless, throughout all the computations made for this paper, we notice some cases where the reducibility of $\Delta_{\Psi}$ seems to implies that Definition~\ref{def:SP}~(2) is not verified. They are: the line of the support are concurrent, and the pivot-points are collinear. We didn't succeeded to prove the implications, and it is possible that they are only a consequence of the low number of lines in the examples studied.

This being said, there is one case where we can state that adding a splitting-polygon will not induce a splitting of the moduli space. Let $\A$ be a line arrangement and let $\Psi$ be a plinth of $\A$. We denote by $E^{\lambda_1}$ and $E^{\lambda_2}$ the two splitting-polygons associated to $\Psi$. In the arrangement $\A^{\lambda_1}=\A\cup E^{\lambda_1}$, $\Psi$ is still a plinth and to add the splitting-polygon $E^{\lambda_2}$ on $\Psi$ will not split the moduli space $\M^\ord(\C_{\Psi,\Psi})$ over $\M^\ord(\C_\Psi)^{\A^{\lambda_1}}$. In other words, to add a second splitting-polygon over a first one with the same plinth will not induce a splitting of the moduli space. Nevertheless, we will see in Section~\ref{sec:successive} that it is possible to add two splitting-polygons (with different plinths) to an arrangement and obtain a moduli space which splits twice.

%%%%%%%%%%%%%%%%%%%%%%%%%%%%%%%%%
\subsection{Splitting-polygon pattern}\mbox{}
%%%%%%%%%%%%%%%%%%%%%%%%%%%%%%%%%

The presence of a splitting-polygon in an arrangement is a strong indicator of a non-connected moduli space. We can thus define the \emph{splitting-polygon pattern} as the combinatorial version the presence of a plinth and a splitting-polygon.

\begin{definition}\label{def:SPP}
	A \emph{splitting-polygon pattern} in a line combinatorics $\C=(\L,\P)$ is formed by three tuples: the \emph{support} $S=(S_1, \dots , S_r) \subset \L$,  the \emph{polygon} $E=(E_1, \dots , E_r) \subset \L$ and the \emph{pivot-points} $(P_1, \dots, P_r) \subset \P_{\geq 3}$ such that
	\begin{enumerate}
		\item $S \cap E = \emptyset$,
		\item for all $ i \in \{1,\dots,r\} $, the cardinal of  $ \{P \in \P_{\geq 3} \mid E_i \in P\} $ is $3$,
		\item for each pivot-point $P_i$, we have $S_i\notin P_i$, $S_{i+1}\notin P_i$ and $P_i \cap E = \{E_i\}$,
		\item for all $ i \in \{1,\dots,r\} $, it exists $Q_i \in \P_3$, such that $Q_i=\{E_{i-1},S_i,E_i\}$, 
	\end{enumerate}
	where all the indices are considered modulo $r$.
\end{definition}

\begin{example}\label{ex:ML_SPP}
	The MacLane line combinatorics $\C_\ML$ is described by $\L=\{L_1,\dots,L_8\}$ and 
	\begin{equation*}
		\P=\left\{
			\begin{array}{c}
				\{ L_1, L_2, L_3 \}, \{ L_1, L_4, L_5 \}, \{ L_1, L_6, L_7 \}, \{ L_1, L_8 \},  \{ L_2, L_4 \}, \{ L_2, L_5, L_7 \}, \\
				 \{L_2, L_6,L_8\}, \{L_3, L_4, L_6 \}, \{ L_3, L_5, L_8 \}, \{ L_3, L_7 \}, \{ L_4, L_7, L_8 \}, \{ L_5, L_6 \}
			\end{array}
		\right\}.
	\end{equation*}
	The tuples $S=(L_1,L_2,L_4)$, $E=(L_6,L_8,L_7)$ and $\P=(\{L_3,L_4,L_6\},\{L_3,L_5,L_8\},\{L_2,L_5,L_7\})$ form a splitting-polygon pattern in the MacLane combinatorics.
\end{example}

%%%%%%%%%%%%%%%%%%%%%%%%%%%%%%%%%
%%%%%%%%%%%%%%%%%%%%%%%%%%%%%%%%%
\section{Applications on arrangements with few lines}\label{sec:examples}
%%%%%%%%%%%%%%%%%%%%%%%%%%%%%%%%%
%%%%%%%%%%%%%%%%%%%%%%%%%%%%%%%%%

In this section, we show that all the small examples of line arrangements with a non-connected moduli space obtained in the classification of Ye~\cite{Ye}, can be constructed using the technique of the splitting-polygon. Naturally, the first example is the MacLane arrangements~\cite{Mac}. They are the smallest arrangements with a non-connected moduli space, it appears that they are also the smallest arrangements which contain a splitting-polygon pattern.  

%%%%%%%%%%%%%%%%%%%%%%%%%%%%%%%%%
\subsection{MacLane arrangements}\label{sec:MacLane}\mbox{}
%%%%%%%%%%%%%%%%%%%%%%%%%%%%%%%%%

Let $\A=\{L_1,\dots,L_5\}$ be the arrangement of $5$ lines which contains two triple points. We assume that these two triple points of $\A$ are $\{L_1,L_2,L_3\}$ and $\{L_1,L_4,L_5\}$.  As we have seen in Example~\ref{ex:plinth}, the tuples $(L_1,L_2,L_4)$ and $(\{L_3,L_4\},\{L_3,L_5\},\{L_2,L_5\})$ form a plinth $\Psi$ of $\A$.

A projective transformation fixes 4 non collinear points. Thus, the action of $\PGL_3(\CC)$ fixes the four double points of $\A$, and, as a consequence of the combinatorics of $\A$, it also fixes the two triples points. In other words, the dimension of $\M^\ord(\A)$ is zero. So by Proposition~\ref{propo:rigidity}, we deduce the $\Psi$ form a rigid projective system of $\M^\ord(\A)$ (which is connected).

Since $\M^\ord(\A)$ is connected and since its dimension is zero, we can work with a representative of the unique class of $\M^\ord(\A)$. Let $\A$ be the arrangement described by the following equations.
\begin{equation*}
	\begin{array}{lll}
		L_1: z=0 \hspace{2cm} & L_2: x=0 \hspace{2cm} & L_3: x-z=0 \\
		L_4: y=0 \hspace{2cm} & L_5: y-z=0 \hspace{2cm} &
	\end{array}
\end{equation*}
Thus, the singular points have the following coordinates.
\begin{equation*}
	\begin{array}{lll}
		\{L_1,L_2,L_3\}: [0:1:0] \hspace{1cm} & \{L_1,L_4,L_5\}: [1:0:0] \hspace{1cm} & \{L_2, L_4\}: [0:0:1] \\
		\{L_2,L_5\}: [0:1:1] \hspace{1cm} & \{L_3, L_4\}: [1:0:1]  \hspace{1cm} & \{L_3,L_5\}: [1:1:1] 
	\end{array}
\end{equation*}

Let $Q^\lambda_1 = [1:\lambda:0]$ be a generic point of $L_1$ (we can assume that the two first coordinates are non-zero since $Q_1^\lambda$ will need to be different of $\{L_1,L_2,L_3\}$ and $\{L_1,_4,L_5\}$). Following the construction made in Section~\ref{sec:construction}, we deduce that the equations of $E^\lambda_1$, $E^\lambda_2$ and $E^\lambda_3$ are respectively $\lambda x + y  - \lambda z = 0$,  $(1-\lambda)x -y + \lambda z = 0$ and $(\lambda - 1)x + \lambda y -\lambda z = 0$. This implies that the polynomial $\Delta_{\Psi}$ is given by:
\begin{equation*}
	\Delta_{\Psi}(\lambda) = 
	\begin{vmatrix}
		0 & \lambda & (\lambda - 1) \\
		0 & 1 & \lambda \\
		1 & -\lambda & -\lambda \\
	\end{vmatrix}
	= \lambda^2 - \lambda + 1.
\end{equation*}
Let $\lambda_1=\frac{1+i\sqrt{3}}{2}$ and $\lambda_2=\frac{1-i\sqrt{3}}{2}$ be the two roots of $\Delta_{\Psi}$. Since $\Delta_{\Psi}$ is irreducible in $\QQ$ (the field of definition of $\A$), then by the first part of Theorem~\ref{thm:irreducible}, we deduce that $\A^{\lambda_1}$ and $\A^{\lambda_2}$ are $\C_\Psi^\ord$-equivalent, and it is easy to check that this shared combinatorics $\C_\Psi^\ord$ is the MacLane combinatorics (see Figure~\ref{fig:MacLane}). Using the second part of Theorem~\ref{thm:irreducible}, we obtain that the ordered moduli space of $\C_\Psi$ splits over $\M^\ord(\A)$. More particularly, we have $\# \M^\ord(\C_\Psi)=2$.

\begin{figure}[h!]
	\begin{tikzpicture}
		\begin{scope}
			\draw (0,-1) -- (0,4) to[out=90, in=225] (1,6) -- (1.5,6.5);
			\node[below] at (0,-1) {$L_2$};
			\draw (2,-1) -- (2,4) to[out=90, in=-45] (1,6) -- (0.5,6.5);
			\node[below] at (2,-1) {$L_3$};
			
			\draw (-3,0) -- (4,0) to[out=0, in=225] (6,1) -- (6.5,1.5);
			\node[left] at (-3,0) {$L_4$};
			\draw (-1,2) -- (4,2) to[out=0, in=135] (6,1) -- (6.5,0.5);
			\node[left] at (-1,2) {$L_5$};
			
			\draw (-1,6) -- (5,6) to[out=0,in=90] (6,5) -- (6,-3);
			\node[left] at (-1,6) {$L_1$};
			
			\node (P1) at (2,0) {$\bullet$};
			\node[above right] at (P1) {$P_{3,4}$};
			\node (P2) at (2,2) {$\bullet$};
			\node[below right] at (P2) {$P_{3,5}$};
			\node (P3) at (0,2) {$\bullet$};
			\node[below right] at (P3) {$P_{2,5}$};
			
			\node (Q1) at (6,-2) {$\circ$};
			\node[below left] at (Q1) {$Q_1^\lambda$};
			\line{(Q1)}{(P1)}{20}{60}{$E_1^\lambda$};
			
			\node (Q2) at (0,1) {$\circ$};
			\node at (0.25,0.55) {$Q_2^\lambda$};
			\line{(P2)}{(Q2)}{50}{130}{$E_2^\lambda$};
			
			\node (Q3) at (-2,0) {$\circ$};
			\node[above left] at (Q3) {$Q_3^\lambda$};
			\line{(Q3)}{(P3)}{40}{220}{$E_3^\lambda$};
			
			\node (R1) at (4,6) {$\circ$};
			\node[above left] at (R1) {$R_1^\lambda$};
			
			\draw[<->, dashed] (R1.23) to[out=40, in =45] (Q1.north east);
			
		\end{scope}
	\end{tikzpicture}
	\caption{Construction of the MacLane arrangement by adding a splitting-triangle on a line arrangements of $5$ lines with two triple points. \label{fig:MacLane}}
\end{figure}
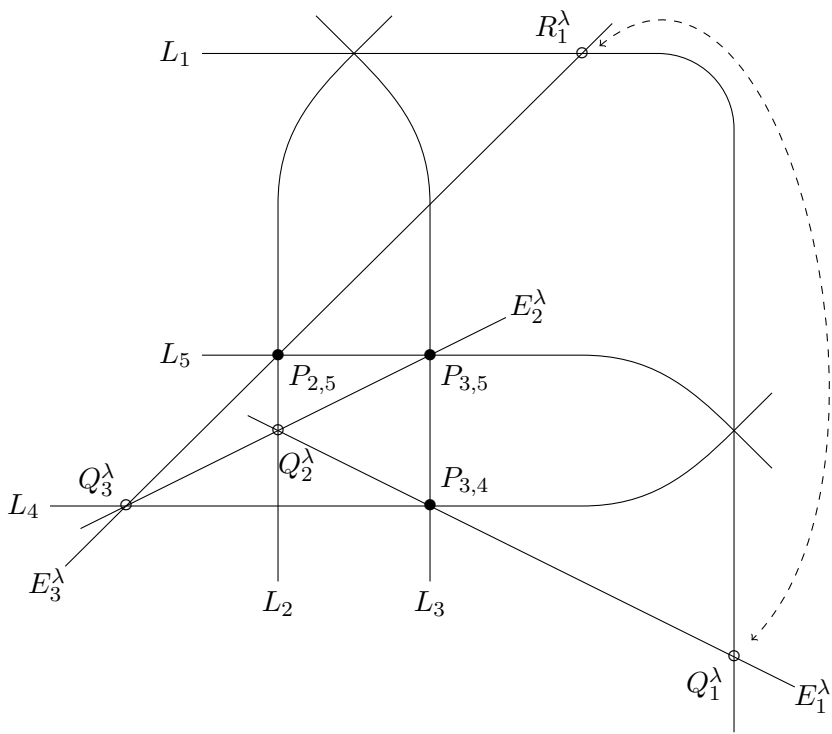

%%%%%%%%%%%%%%%%%%%%%%%%%%%%%%%%%
\subsection{Nine lines arrangements}\mbox{}
%%%%%%%%%%%%%%%%%%%%%%%%%%%%%%%%%

According to the classification of the complement homeomorphism types of the 9 lines arrangements obtained by Ye in~\cite{Ye}, there are 4 types of 9 lines arrangements: arrangements whose moduli space is irreducible (and so connected), arrangements which contain a MacLane arrangement, the Falk-Sturmfels arrangements and the Nazir-Yoshinaga arrangements. The case of the MacLane arrangements have been studied in the previous section. We will focus here on the two last types. 

We consider the arrangement $\A$ of 6 lines defined over $\QQ$ given by the following equations.
\begin{equation*}
	\begin{array}{lll}
	 L_1: z=0  \hspace{2cm} & L_2: x=0 \hspace{2cm} & L_3: x-z=0 \\
	L_4: y=0  \hspace{2cm} & L_5: y-z=0 \hspace{2cm} & L_6: x-y=0
	\end{array}
\end{equation*}
As seen in the previous section, the action of $\PGL_3(\CC)$ fixes the lines $L_1,\dots,L_5$. Since $L_6$ can be define from the intersection points of the fifth first lines, then the action of $\PGL_3(\CC)$ also fixes $L_6$. This implies that the moduli space of $\A$ has dimension~0.

The difference between the Falk-Sturmfels and the Nazir-Yoshinaga arrangements will be made from the plinth considered. Additionally, their combinatorics can be distinguished from the MacLane combinatorics by the structure of their plinths. Indeed, in the following cases the plinths considered require $6$ lines, while the MacLane arrangements plinth requires only $5$ lines.

\subsubsection{Falk-Sturmfels arrangements}
We consider here the plinth $\Psi_1$ of $\A$ whose support is the tuple $(L_1,L_2,L_5)$ and pivot-points are $(\{L_3,L_4\},\{L_1,L_6\},\{L_2,L_4,L_6\})$. Following the construction of Section~\ref{sec:construction}, let $Q_1^\lambda=[1:\lambda:0]$ be a generic point of $L_1$, then the lines $E^\lambda_1$, $E^\lambda_2$ and $E^\lambda_3$ are respectively given by the equations: $ -\lambda x + y + \lambda z=0$, $x - y - \lambda z =0$ and $-x + (\lambda +1) y =0$. This implies that the polynomial $\Delta_\Psi$ is $\lambda^2+\lambda-1$, which is irreducible in $\QQ$. So, by Theorem~\ref{thm:irreducible}, if $\lambda_1$ and $\lambda_2$ are the two roots of $\Delta_{\Psi}$, then $\A^{\lambda_1}$ and $\A^{\lambda_2}$ are representatives of the two connected components of $\M^\ord(\C_{\Psi_1})$. These arrangements admits 8 triples points and 1 quadruple points, and are the Falk-Sturmfels arrangements. They are the unique 9 lines arrangements defined in a real quadratic number field.

\subsubsection{Nazir-Yoshinaga arrangements}
In this case, we consider the plinth $\Psi_2$ of $\A$ whose support is $(L_1,L_2,L_4)$ and pivot-points are $(\{L_3,L_4\},\{L_1,L_6\},\{L_2,L_5\})$. We consider again the generic point $Q_1^\lambda=[1:\lambda:0]$ in $L_1$. The line $E^\lambda_1$, $E^\lambda_2$ and $E^\lambda_3$ are respectively given by the equations  $ -\lambda x + y + \lambda z=0$, $x - y - \lambda z =0$ and $x + \lambda y - \lambda z = 0$. Thus, the polynomial $\Delta_\Psi$ is $\lambda^2+1$, which is irreducible in $\QQ$. Remark here that the equations of the Falk-Sturmfels arrangements and those of the Nazir-Yoshinaga arrangements only different by the last line and the definition field. Using Theorem~\ref{thm:irreducible}, we obtain that if $\lambda_1$ and $\lambda_2$ are the two roots of $\Delta_{\Psi}$, then $\A^{\lambda_1}$ and $\A^{\lambda_2}$ are representatives of the two connected components of $\M^\ord(\C^{\Psi_2})$. These arrangements admits only triples points (10 in total), this allows to distinguish their combinatorics from those of the Falk-Sturmfels arrangements. They are the Nazir-Yoshinaga arrangements since they are define over $\QQ[i]$.

\begin{rmk}
	A case-by-case study can prove that, using the technique of the splitting-polygon, we cannot construct other types of $9$ line arrangements with non-connected moduli space. This is in accordance with the classification of Ye~\cite{Ye}.
\end{rmk}

%%%%%%%%%%%%%%%%%%%%%%%%%%%%%%%%%
\subsection{A splitting-quadrilateral}\mbox{}
%%%%%%%%%%%%%%%%%%%%%%%%%%%%%%%%%

In all the previous examples, we used only splitting-triangles (i.e splitting-polygon with $r=3$). So let illustrate the notion of splitting-polygon for $r=4$. We consider the arrangement $\A$ formed by the following equations.
%  and illustrated in Figure~\ref{fig:7lines}.
\begin{equation*}
	\begin{array}{llll}
	L_1: z=0 \hspace{1cm} 
	& L_2: x=0 \hspace{1cm}
	& L_3: x-z=0\hspace{1cm}
	& L_4: y=0 \\
	L_5: y-z=0 \hspace{1cm}
	& L_6: x-y=0 \hspace{1cm}
	& L_7: x-y+z=0 \hspace{1cm}
	&
	\end{array}
\end{equation*}

Let $\Psi$ be the plinth defined by the tuples $(\{L_3,L_7\},\{L_1,L_6,L_7\},\{L_2,L_4,L_6\},\{L_3,L_4\})$ for the pivot-points and by $(L_2,L_4,L_3,L_5)$ for the support. We choose a generic point $Q_1^\lambda=[0 : 1 : \lambda]$ in the line~$L_1$. Following the construction of Section~\ref{sec:construction}, we have:
\begin{align*}
	& E_1^\lambda: (2\lambda -1) x - \lambda y + z =0 \quad & 
	& E_2^\lambda: (2\lambda -1) x - (2\lambda -1) y + z = 0 \\
	& E_3^\lambda: -2 \lambda x + (2\lambda -1) y = 0 \quad & 
	& E_4^\lambda: 2 \lambda x + y - 2\lambda z = 0
\end{align*}
The polynomial $\Delta_\Psi$ is $2\lambda^2 - 1$, which is irreducible over $\QQ$ (the field of definition of $\A$). Then, by Theorem~\ref{thm:irreducible}, for $\lambda_1=\sqrt{2}/2$ and $\lambda_2=-\sqrt{2}/2$ (the roots of $\Delta_\Psi$), the lines $(E_1^{\lambda_i},\dots,E_4^{\lambda_i})$ form a splitting-quadrilateral on the plinth $\Psi$. We can remark that the roots of the polynomial $\Delta_\Psi$ are real, we can thus picture the two arrangements $\A^{\lambda_1}$ and $\A^{\lambda_2}$ (see Figure~\ref{fig:splitting_quadri}).

\begin{figure}[h!]
	\centering
	\includegraphics[scale=0.8]{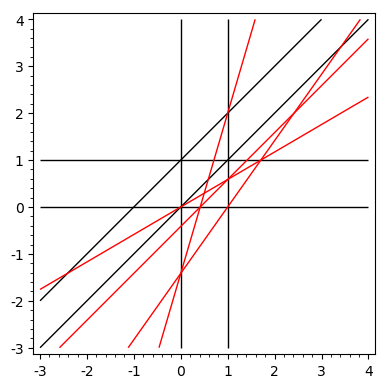}
	\hspace{0.25cm}
	\includegraphics[scale=0.8]{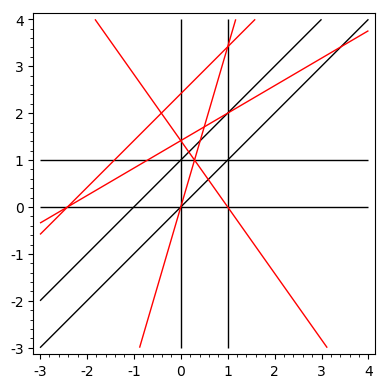}
	\caption{The arrangements $\A^{\lambda_1}$ (left), and $\A^{\lambda_2}$ (right), the line $L_1$ is the line at infinity. In red are pictured the splitting-quadrilateral. \label{fig:splitting_quadri} }
\end{figure}

\begin{rmk}
	A very similar construction is underlying the arrangements with a non-connected moduli spaces produced by the author and Viu-Sos in~\cite{GueViu:configurations}. This paper, and the discovery of the arrangements given in Section~\ref{sec:successive}, were the starting point of the present work.
\end{rmk}

%%%%%%%%%%%%%%%%%%%%%%%%%%%%%%%%%
%%%%%%%%%%%%%%%%%%%%%%%%%%%%%%%%%
\section{Successive splittings}\label{sec:successive}
%%%%%%%%%%%%%%%%%%%%%%%%%%%%%%%%%
%%%%%%%%%%%%%%%%%%%%%%%%%%%%%%%%%

In this section, we explore some examples of arrangements constructed by adding two successive splitting-polygons (with different plinth due to the last paragraph of Section~\ref{sec:construction}). This allows to construct arrangements whose moduli space has 4 connected components. The first such example is naturally the Rybnikov arrangements.

%%%%%%%%%%%%%%%%%%%%%%%%%%%%%%%%%
\subsection{Rybnikov arrangements}\mbox{}
%%%%%%%%%%%%%%%%%%%%%%%%%%%%%%%%%

As we have seen in Section~\ref{sec:MacLane}, the MacLane arrangements can be constructed from the arrangement $\A$ formed by 5 lines with two triple points. In this section, we will see that it is easy to construct the Rybnikov arrangements~\cite{Ryb} using the splitting-polygon technique twice.

Let $\A$ be the arrangement of $7$ lines which contains $3$ triple points on $L_1$: $\{L_1,L_2,L_3\}$, $\{L_1,L_4,L_5\}$ and $\{L_1,L_6,L_7\}$, and only double points outside. Following the construction made in Section~\ref{sec:MacLane}, we can add a splitting-polygon on the plinth $\Psi_1$ whose support is $(L_1,L_2,L_4)$ and pivot-points are $(\{L_3,L_4\},\{L_3,L_5\},\{L_2,L_5\})$. At this first step, the polynomial $\Delta_{\Psi_1}$ is irreducible, thus by Theorem~\ref{thm:irreducible}, the arrangements $\A^{\lambda_1}$ and $\A^{\lambda_2}$ are representatives of the two connected components of $\M^\ord(\C_{\Psi_1})$, and are defined over $\FF=\QQ[\zeta]$, where $\zeta$ is a root of $\lambda^2-\lambda+1$. 

This construction produces arrangements which contain a MacLane arrangement (formed by the lines $L_1,\dots,L_5$ together with the lines $E_1^\lambda, E_2^\lambda, E_3^\lambda$). In order to construct the Rybnikov arrangements, we need to produce a second MacLane arrangement from $\A^{\lambda_i}$. To do so, we will consider another plinth $\Psi_2$ of $\A$ (and thus of $\C_{\Psi_1	}$ too) which is, in $\A$, equivalent\footnote{Here by equivalent, we mean that it exists an automorphism of the combinatorics $\C(\A)$ which sent $\Psi_1$ on $\Psi_2$} to $\Psi_1$ but not in $\C_{\Psi_1}$. This plinth is given by $(L_1,L_2,L_6)$ for the support and by $(\{L_3,L_6\}, \{L_3,L_7\}, \{L_2,L_7\})$ the the pivot-points. As already announced at the end of Section~\ref{sec:construction}, the polynomial $\Delta_{\Psi_2}$ is reducible in the definition field $\FF$ of $\A^{\lambda_i}$. Indeed, the conditions imposed by the plinth $\Psi_2$ are projectively equivalent to those imposed by the plinth $\Psi_1$. So it seems natural that the decomposition-field of $\Delta_{\Psi_2}$ is the same as those of $\Delta_{\Psi_1}$, and so is $\FF$ too. Nevertheless, if the line $L_6$ and $L_7$ are generic enough with the lines $L_4$ and $L_5$, then the tuples of lines $(E_1^{\lambda_i,\lambda_j}, E_2^{\lambda_i,\lambda_j}, E_3^{\lambda_i,\lambda_j})$, for $j\in\{1,2\}$, form two distinct splitting-polygons on the plinth $\Psi_2$. Thus, by Theorem~\ref{thm:SP}, we obtain that the moduli space $\M^\ord(\C^{\Psi_1,\Psi_2}; \M^\ord(\C_\Psi)^{\A^{\lambda_i}})$, for $j\in\{1,2\}$, splits over $\M^\ord(\C_{\Psi_1})^{\A^{\lambda_i}}$. As a consequence, we have that $\# \M^\ord(\C^{\Psi_1,\Psi_2}) = 4$.

%%%%%%%%%%%%%%%%%%%%%%%%%%%%%%%%%
\subsection{MacLane splittings}\mbox{}
%%%%%%%%%%%%%%%%%%%%%%%%%%%%%%%%%

The addition of a splitting-polygon on an arrangement which contains the MacLane arrangement in order to construct an arrangement whose moduli space has four connected components have been successfully use in the previous section to re-construct the Rybnikov arrangements. In that construction, the plinth considered to support the second splitting-polygon is, in fact, a plinth of the original arrangement of $7$ lines. However, we could have consider a plinth which contains some part of the first splitting-polygon. In this section, we will explore some cases of splitting-polygons added directly on the MacLane arrangements. This construction will produce line arrangements of $11$ lines whose moduli space have four connected components. Arithmetically, these arrangements will be between the Rybnikov arrangements which are not Galois conjugated in their field of definition, and the arrangements obtain by the author in~\cite{Gue:ZP,Gue:LLN} which are arithmetic quadruples in the $5$th cyclotomic field. Indeed, they will be arithmetic quadruples, but with the Klein group as Galois group of their field of definition.

We consider $\ML_1$ and $\ML_2$: the two MacLane arrangements as constructed in Section~\ref{sec:MacLane}. Their field of definition is the number field $\QQ[\zeta]$, where $\zeta$ is a root of $\lambda^2 - \lambda +1$. We denote by $\C_\ML$ their combinatorics. By Proposition~\ref{prop:dimension}, the dimension of $\M^\ord(\C_\ML)$ is zero. Thus, by Proposition~\ref{propo:rigidity},  any subsets $\{L_{i_1},\dots,L_{i_r}\} \subset \ML_k$ and $\{P_{j_1},\dots,P_{j_r}\}\in\sing(\ML_k)$ form a rigid projective system of $\M^\ord(\C_\ML)^{\ML_k}$. They can thus be used as a plinth to support a splitting-polygon (as soon as the conditions of Definition~\ref{def:SP} are verified). Up to our computation, there are 56 different ways to add a splitting-triangle on the MacLane arrangements. In this list of arrangements, six of them kept our attention for a particular topological property (see Theorem~\ref{thm:pi_1}). 

Consider three line $S=(L_{i_1},L_{i_2},L_{i_3})$ of $\C_\ML$ which are not concurrent. Let $(P_1,P_2,P_3)$ be three points which form, with $S$, a plinth $\Psi$ of $\C_\ML$. In addition, we assume that $L_{i_{k-1}}\in P_i$ (where the indices of $i$ are considered modulo $3$), and that the $P_i$'s are not collinear.  Up to automorphism of the combinatorics, it exists $6$ such plinths which give rise to a splitting-polygon. These particular plinths are listed in Table~\ref{tab:plinths}. The first exponent in the name of the arrangements corresponds to the first splitting-polygon (and so on which MacLane arrangement it is built), while the second is associated to the second splitting-polygon. In each case, the polynomial $\Delta_{\Psi}$ is irreducible in $\QQ[\zeta]$, thus these arrangements are defined over a number field of degree 4. To change $i$ by $i+1 \mod 2$ (resp. $j$ by $j+1 \mod 2$) corresponds to take the complex conjugate of the equations of the first (resp the second) splitting-polygon.

\begin{table}[h!]
	\begin{equation*}
		\def\arraystretch{1.2}
		\begin{array}{|c|c|c|}
			\hline
			\text{Support} & \text{Pivot-points} & \text{Arrangement's name} \\ \hline
			(1, 2, 4) & (\{3, 4, 6\}, \{1, 6, 8\}, \{2, 5, 8\}) & \ML_1^{i,j} \\ \hline
			(1, 2, 4) & (\{4, 7, 8\}, \{1, 6, 8\}, \{2, 6, 7\}) & \ML_2^{i,j} \\ \hline
			(1, 2, 4) & (\{4, 7, 8\}, \{1, 7\}, \{2, 5, 8\}) & \ML_3^{i,j} \\ \hline
			(1, 2, 5) & (\{5, 6\}, \{1, 6, 8\}, \{2, 4\}) & \ML_4^{i,j} \\ \hline
			(1, 2, 5) & (\{3, 5, 7\}, \{1, 6, 8\}, \{2, 6, 7\}) & \ML_5^{i,j} \\ \hline
			(1, 2, 5) & (\{5, 6\}, \{1, 7\}, \{2, 6, 7\}) & \ML_6^{i,j} \\ \hline
		\end{array}
	\end{equation*}
	\caption{Particular plinths of the MacLane combinatorics.\label{tab:plinths}}
\end{table}

Following the definition of Marco in~\cite{Mar:pencils}, all these arrangements are homologically rigid (this depends only on the combinatorics $\C_{\ML,\Psi}$). This means that, if it exists an isomorphism between $\pi_1(\CC\PP^2\setminus \ML_k^{i,j})$ and  $\pi_1(\CC\PP^2\setminus \ML_k^{i',j'})$, for $i,j,i',j'\in\{1,2\}$, then it induces the identity on the Abelianization of these groups. By applying the Alexander Invariant test of level 2 (described in~\cite{ACCM:Rybnikov,ACGM:arithmetic}), we obtain the following theorem. We don't give here more details about the proof, since the strategy is the exactly the same as in \cite{ACCM:Rybnikov,ACGM:arithmetic,Gue:LLN}, and since the author used the same program to perform the computation. We refer to these articles for more details (the construction of the Alexander Invariant test is done in~\cite{ACCM:Rybnikov} while a Sagemath code is given in~\cite{ACGM:arithmetic}).
	
\begin{theorem}\label{thm:pi_1}
	For a fixed $k\in\{1,\dots,6\}$, if $ i+j \not\equiv i'+j' \mod 2$ then
	\begin{equation*}
		\pi_1(\CC\PP^2\setminus \ML_k^{i,j}) \not\simeq \pi_1(\CC\PP^2\setminus \ML_k^{i',j'}).
		\
	\end{equation*}
\end{theorem}

To our knowledge, these arrangements never appear in the literature before. Due to their particular arithmetic property, it would be interesting to see if the invariants developed by Bannai, Shirane and Tokunaga~\cite{Ban:splitting,Shi:splitting,Tok:elliptic} could distinguish their topology. Furthermore, neither the linking-invariants~\cite{AFG:invariant,Gue:LLN} nor the torsion order of the first lower central series quotients of their fundamental groups~\cite{Suc:enumerative,AGV:torsion} can distinguish it.

%%%%%%%%%%%%%%%%%%%%%%%%%%%%%%%%%
%%%%%%%%%%%%%%%%%%%%%%%%%%%%%%%%%
\section{Acknowledgments}
%%%%%%%%%%%%%%%%%%%%%%%%%%%%%%%%%
%%%%%%%%%%%%%%%%%%%%%%%%%%%%%%%%%

The author is supported by a post-doctoral grant associated to the JSPS KAKENHI Grant Number 17H06128 (Grant-in-Aid for Scientific Research (S)) leads by Prof. Saeki.

%%%%%%%%%%%%%%%%%%%%%%%%%%%%%%
%        Bibliography        %
%%%%%%%%%%%%%%%%%%%%%%%%%%%%%%

\bibliographystyle{plain}
\bibliography{biblio}

\begin{thebibliography}{10}

\bibitem{Abe:addition-deletion}
Takuro {Abe}.
\newblock Addition-deletion theorem for free hyperplane arrangements and
  combinatorics.
\newblock Available at \texttt{arXiv:1811.03780}, 2018.

\bibitem{Abe:deletion}
Takuro {Abe}.
\newblock {Deletion theorem and combinatorics of hyperplane arrangements}.
\newblock {\em {Math. Ann.}}, 373(1-2):581--595, 2019.

\bibitem{ACCM:topology}
Enrique {Artal Bartolo}, Jorge {Carmona Ruber}, Jos\'e~Ignacio {Cogolludo
  Agust\'{\i}n}, and Miguel {Marco Buzun\'ariz}.
\newblock {Topology and combinatorics of real line arrangements}.
\newblock {\em {Compos. Math.}}, 141(6):1578--1588, 2005.

\bibitem{ACCM:Rybnikov}
Enrique {Artal Bartolo}, Jorge {Carmona Ruber}, Jos\'e~Ignacio {Cogolludo
  Agust\'{\i}n}, and Miguel~\'Angel {Marco Buzun\'ariz}.
\newblock {Invariants of combinatorial line arrangements and Rybnikov's
  example}.
\newblock In {\em Singularity theory and its applications. Papers from the 12th
  MSJ International Research Institute of the Mathematical Society of Japan,
  Sapporo, Japan, September 16--25, 2003}, pages 1--34. Tokyo: Mathematical
  Society of Japan, 2006.

\bibitem{ACGM:arithmetic}
Enrique {Artal Bartolo}, Jos\'e~Ignacio {Cogolludo-Agust\'{\i}n}, Beno\^{\i}t
  {Guerville-Ball\'e}, and Miguel {Marco-Buzun\'ariz}.
\newblock {An arithmetic Zariski pair of line arrangements with non-isomorphic
  fundamental group}.
\newblock {\em {Rev. R. Acad. Cienc. Exactas F\'{\i}s. Nat., Ser. A Mat.,
  RACSAM}}, 111(2):377--402, 2017.

\bibitem{Ban:splitting}
Shinzo {Bannai}.
\newblock {A note on splitting curves of plane quartics and multi-sections of
  rational elliptic surfaces}.
\newblock {\em {Topology Appl.}}, 202:428--439, 2016.

\bibitem{AFG:invariant}
Enrique~Artal {Bartolo}, Vincent {Florens}, and Beno\^{\i}t
  {Guerville-Ball\'e}.
\newblock {A topological invariant of line arrangements}.
\newblock {\em {Ann. Sc. Norm. Super. Pisa, Cl. Sci. (5)}}, 17(3):949--968,
  2017.

\bibitem{AGV:torsion}
Enrique~Artal {Bartolo}, Beno\^{\i}t {Guerville-Ball\'e}, and Juan {Viu-Sos}.
\newblock {Fundamental groups of real arrangements and torsion in the lower
  central series quotients}.
\newblock {\em {Exp. Math.}}, 29(1):28--35, 2020.

\bibitem{CohSuc:braid}
Daniel~C. {Cohen} and Alexander~I. {Suciu}.
\newblock {The braid monodromy of plane algebraic curves and hyperplane
  arrangements}.
\newblock {\em {Comment. Math. Helv.}}, 72(2):285--315, 1997.

\bibitem{Falk:minimal}
Michael {Falk}.
\newblock {The minimal model of the complement of an arrangement of
  hyperplanes}.
\newblock {\em {Trans. Am. Math. Soc.}}, 309(2):543--556, 1988.

\bibitem{FalkRan:homotopy}
Michael {Falk} and Richard {Randell}.
\newblock {On the homotopy theory of arrangements. II}.
\newblock In {\em Arrangements -- Tokyo 1998. Proceedings of a workshop on
  mathematics related to arrangements of hyperplanes, Tokyo, Japan, July
  13--18, 1998. In honor of the 60th birthyear of Peter Orlik}, pages 93--125.
  Tokyo: Kinokuniya Company Ltd., 2000.

\bibitem{Gue:ZP}
Beno\^{\i}t {Guerville-Ball\'e}.
\newblock {An arithmetic Zariski 4-tuple of twelve lines}.
\newblock {\em {Geom. Topol.}}, 20(1):537--553, 2016.

\bibitem{Gue:LLN}
Beno{\^i}t {Guerville-Ball{\'e}}.
\newblock The loop-linking numbers of line arrangements.
\newblock Available at \texttt{arXiv:2004.03550}, 2020.

\bibitem{GueViu:configurations}
Beno\^{\i}t {Guerville-Ball\'e} and Juan {Viu-Sos}.
\newblock {Configurations of points and topology of real line arrangements}.
\newblock {\em {Math. Ann.}}, 374(1-2):1--35, 2019.

\bibitem{Lib}
A.~{Libgober}.
\newblock {Characteristic varieties of algebraic curves.}
\newblock In {\em Applications of algebraic geometry to coding theory, physics
  and computation. Proceedings of the NATO advanced research workshop, Eilat,
  Israel, February 25--March 1, 2001}, pages 215--254. Dordrecht: Kluwer
  Academic Publishers, 2001.

\bibitem{Mac}
S.~{MacLane}.
\newblock {Some interpretations of abstract linear dependence in terms of
  projective geometry.}
\newblock {\em {Am. J. Math.}}, 58:236--240, 1936.

\bibitem{Mar:pencils}
Miguel~\'Angel {Marco Buzun\'ariz}.
\newblock {A description of the resonance variety of a line combinatorics via
  combinatorial pencils}.
\newblock {\em {Graphs Comb.}}, 25(4):469--488, 2009.

\bibitem{NazYos}
Shaheen {Nazir} and Masahiko {Yoshinaga}.
\newblock {On the connectivity of the realization spaces of line arrangements}.
\newblock {\em {Ann. Sc. Norm. Super. Pisa, Cl. Sci. (5)}}, 11(4):921--937,
  2012.

\bibitem{OrlSol}
Peter {Orlik} and Louis {Solomon}.
\newblock {Combinatorics and topology of complements of hyperplanes}.
\newblock {\em {Invent. Math.}}, 56:167--189, 1980.

\bibitem{OrlTer}
Peter {Orlik} and Hiroaki {Terao}.
\newblock {\em {Arrangements of hyperplanes}}, volume 300.
\newblock Berlin: Springer-Verlag, 1992.

\bibitem{Ryb}
G.~L. {Rybnikov}.
\newblock {On the fundamental group of the complement of a complex hyperplane
  arrangement}.
\newblock {\em {Funct. Anal. Appl.}}, 45(2):137--148, 2011.

\bibitem{Shi:splitting}
Taketo {Shirane}.
\newblock {A note on splitting numbers for Galois covers and
  \(\pi_1\)-equivalent Zariski \(k\)-plets}.
\newblock {\em {Proc. Am. Math. Soc.}}, 145(3):1009--1017, 2017.

\bibitem{Suc:enumerative}
Alexander~I. {Suciu}.
\newblock {Fundamental groups of line arrangements: Enumerative aspects}.
\newblock In {\em Advances in algebraic geometry motivated by physics.
  Proceedings of the AMS special session on enumerative geometry in physics,
  University of Massachusetts, Lowell, MA, USA, April 1--2, 2000}, pages
  43--79. Providence, RI: American Mathematical Society (AMS), 2001.

\bibitem{Ter}
Hiroaki {Terao}.
\newblock {The exponents of a free hypersurface}.
\newblock {Singularities, Summer Inst., Arcata/Calif. 1981, Proc. Symp. Pure
  Math. 40, Part 2, 561-566 (1983).}, 1983.

\bibitem{Tok:elliptic}
Hiro-O {Tokunaga}.
\newblock {Sections of elliptic surfaces and Zariski pairs for conic-line
  arrangements via dihedral covers}.
\newblock {\em {J. Math. Soc. Japan}}, 66(2):613--640, 2014.

\bibitem{Ye}
Fei {Ye}.
\newblock {Classification of moduli spaces of arrangements of nine projective
  lines}.
\newblock {\em {Pac. J. Math.}}, 265(1):243--256, 2013.

\bibitem{Zas}
Thomas {Zaslavsky}.
\newblock {\em {Facing up to arrangements: Face-count formulas for partitions
  of space by hyperplanes}}, volume 154.
\newblock Providence, RI: American Mathematical Society (AMS), 1975.

\end{thebibliography}

\end{document}